\documentclass[11pt,oneside,english]{article}
\usepackage[T1]{fontenc}
\usepackage[latin9]{inputenc}
\usepackage{geometry}
\usepackage{amstext}
\usepackage{amsthm}
\usepackage{amsmath,amsfonts,amsthm,amssymb,mathrsfs,dsfont, xspace} 
\usepackage{mathtools}
\usepackage{graphicx}
\usepackage[colorinlistoftodos,textsize=scriptsize,backgroundcolor=orange!20]{todonotes}
\usepackage[colorlinks=true, allcolors=blue]{hyperref}
\usepackage[capitalize,nameinlink]{cleveref}
\usepackage{verbatim}
\usepackage{enumerate}

\allowdisplaybreaks
\DeclareFontFamily{OT1}{pzc}{}
\DeclareFontShape{OT1}{pzc}{m}{it}{<-> s * [1.10] pzcmi7t}{}
\DeclareMathAlphabet{\mathpzc}{OT1}{pzc}{m}{it}
\usepackage{bbm}

\makeatletter
\newtheorem{theorem}{Theorem}[section]

\newtheorem*{namedtheorem}{\theoremname}
\newcommand{\theoremname}{testing}

\newtheorem{lemma}[theorem]{Lemma}

\newtheorem{proposition}[theorem]{Proposition}
\newtheorem{observation}[theorem]{Observation}

\newtheorem{corollary}[theorem]{Corollary}

\newtheorem{conjecture}[theorem]{Conjecture}
\newtheorem*{question*}{Question}

\theoremstyle{definition}
\newtheorem{definition}[theorem]{Definition}

\newtheorem{remark}[theorem]{Remark}

\theoremstyle{plain}

\makeatother

\usepackage{babel}

\providecommand{\theoremname}{Theorem}

\DeclareMathOperator{\adj}{adj}
\DeclareMathOperator{\supp}{supp}
\DeclareMathOperator{\0}{\textbf{0}}

\DeclareFontFamily{U}{mathc}{}
\DeclareFontShape{U}{mathc}{m}{it}%
{<->s*[1.03] mathc10}{}
\DeclareMathAlphabet{\mathscr}{U}{mathc}{m}{it}

\title{Singularity of random symmetric matrices -- a combinatorial approach to improved bounds}
\author{Asaf Ferber \thanks{Massachusetts Institute of Technology. Department of Mathematics. Email: {\tt ferbera@mit.edu}. Research is partially supported by NSF 6935855.} \and 
Vishesh Jain\thanks{Massachusetts Institute of Technology. Department of Mathematics. Email: {\tt visheshj@mit.edu}. Research is partially supported by NSF CCF 1665252,  NSF DMS-1737944 and ONR N00014-17-1-2598.
}}
\date{}

\begin{document}
\maketitle
\global\long\def\R{\mathbb{R}}

\global\long\def\S{\mathbb{S}}

\global\long\def\Z{\mathbb{Z}}

\global\long\def\C{\mathbb{C}}

\global\long\def\Q{\mathbb{Q}}

\global\long\def\N{\mathbb{N}}

\global\long\def\P{\mathbb{P}}

\global\long\def\F{\mathbb{F}}

\global\long\def\U{\mathcal{U}}

\global\long\def\V{\mathcal{V}}

\global\long\def\E{\mathbb{E}}

\global\long\def\Ev{\mathscr{Rk}}

\global\long\def\Dg{\mathscr{D}}

\global\long\def\Ndg{\mathscr{ND}}

\global\long\def\Rv{\mathcal{R}}

\global\long\def\Gv{\mathscr{Null}}

\global\long\def\Hv{\mathscr{Orth}}

\global\long\def\Supp{{\bf Supp}}

\global\long\def\Sv{\mathscr{Spt}}

\global\long\def\ring{\mathfrak{R}}

\global\long\def\1{\mathbbm{1}}

\global\long\def\Bad{{\bf B}}

\global\long\def\A{\mathcal{A}}

\global\long\def\L{\mathcal{L}}

\begin{abstract}
Let $M_n$ denote a random symmetric $n \times n$ matrix whose upper diagonal entries are independent and identically distributed Bernoulli random variables (which take values $1$ and $-1$ with probability $1/2$ each). It is widely conjectured that $M_n$ is singular with probability at most $(2+o(1))^{-n}$. On the other hand, the best known upper bound on the singularity probability of $M_n$, due to Vershynin (2011), is $2^{-n^c}$, for some unspecified small constant $c > 0$. This improves on a polynomial singularity bound due to Costello, Tao, and Vu (2005), and a bound of Nguyen (2011) showing that the singularity probability decays faster than any polynomial. In this paper, improving on all previous results, we show that the probability of singularity of $M_n$ is at most $2^{-n^{1/4}\sqrt{\log{n}}/1000}$ for all sufficiently large $n$. The proof utilizes and extends a novel combinatorial approach to discrete random matrix theory, which has been recently introduced by the authors together with Luh and Samotij.
\end{abstract}

\section{Introduction}
The invertibility problem for Bernoulli matrices is one of the most well-studied problems in discrete random matrix theory. Letting $A_n$ denote a random $n\times n$ matrix, whose entries are independent and identically distributed (i.i.d.) Bernoulli random variables which take values $\pm 1$ with probability $1/2$ each, this problem asks for the value of $c_n$, which is the probability that $A_n$ is singular. By considering the event that two rows or two columns of $A_n$ are equal (up to a sign), it is clear that $$c_n \geq (1+o(1))n^{2}2^{1-n}.$$ It has been widely conjectured that this bound is, in fact, tight. On the other hand, perhaps surprisingly, it is non-trivial even to show that $c_n$ tends to $0$ as $n$ goes to infinity -- this was first accomplished in 1967 by Koml\'os \cite{komlos1967determinant}, who showed using the classical Erd\H{o}s-Littlewood-Offord anti-concentration inequality that $$c_n = O\left(n^{-1/2}\right).$$ Subsequently, a breakthrough result due to  Kahn, Koml\'os, and Szemer\'edi in 1995 \cite{kahn1995probability} showed that $$c_n = O(0.999^{n}).$$
After intermediate improvements in the base of the exponent due to Tao and Vu \cite{tao2007singularity} and  Bourgain, Vu, and Wood
\cite{bourgain2010singularity}, this conjecture has been settled up to lower order terms recently (in fact, a few months after the appearance of the present work) in a very impressive work of Tikhomirov \cite{Tikhomirov2018}, showing that $$c_n \leq (2+o(1))^{-n}.$$

Another widely studied model of random matrices is that of random \emph{symmetric} matrices; apart from being important for applications, it is also very interesting from a technical perspective as  it is one of the simplest models with nontrivial correlations between the entries of the matrix. Formally, let $M_n$ denote a random $n\times n$ symmetric matrix, whose upper-diagonal entries are i.i.d. Bernoulli random variables which take values $\pm 1$ with probability $1/2$ each, and let $q_n$ denote the probability that $M_n$ is singular. Despite its similarity to $c_n$, much less is known about $q_n$, as we discuss below.  

The problem of determining whether $q_n$ tends to $0$ as $n$ goes to infinity was first posed by Weiss in the early 1990s and only settled in 2005 by Costello, Tao, and Vu \cite{costello2006random}, who showed that $$q_n = O\left(n^{-1/8 + o(1)}\right).$$ In order to do this, they introduced and studied a quadratic variant of the Erd\H{o}s-Littlewood-Offord inequality.  
Subsequently, Nguyen \cite{nguyen2012inverse} developed a quadratic variant of \emph{inverse} Littlewood-Offord theory to show that $$q_n = O_{C}(n^{-C})$$ for any $C>0$, where the implicit constant in $O_{C}(\cdot)$ depends only on $C$. This so-called quadratic inverse Littlewood-Offord theorem in \cite{nguyen2012inverse} builds on previous work of Nguyen and Vu \cite{nguyen2011optimal}, which is itself based on deep Freiman-type theorems in additive combinatorics (see \cite{tao2008john} and the references therein).
The current best known upper bound on $q_n$ is due to Vershynin \cite{vershynin2014invertibility}, who used a sophisticated and technical geometric framework pioneered by Rudelson and Vershynin \cite{rudelson2008littlewood,rudelson2010non} to show that $$q_n = O(2^{-n^c})$$ for some unspecified small constant $c > 0$. 

As far as lower bounds on $q_n$ are concerned, once again, by considering the event that the first and last rows of $M_n$ are equal (up to a sign), we see that $q_n \geq (2+o(1))^{-n}$. It is commonly believed that this lower bound is tight. 
\begin{conjecture}[\cite{costello2006random,vu2008random}]
\label{conjecture:prob-singularity}
We have
$$q_n = (2+o(1))^{-n}.$$
\end{conjecture}

In this paper, we obtain a much stronger upper bound on $q_n$, thereby making progress towards \cref{conjecture:prob-singularity}.  
\begin{theorem}
\label{thm:main-thm}
There exists $n_0 \in \N$ such that for all $n\geq n_0$, 
$$q_n \leq 2^{-n^{1/4}\sqrt{\log{n}}/1000}.$$
\end{theorem}

\begin{remark}
While the constant $1000$ in the above theorem is somewhat arbitrary, the leading order term $n^{1/4}\sqrt{\log{n}}$ in the exponent is optimal for the argument in this paper. We believe that improving the exponent to even $n^{(1/2) + \epsilon}$ (for some absolute constant $\epsilon > 0$) will likely require new ideas beyond those in the present work, since even in the case of i.i.d. Rademacher random matrices, the combinatorial techniques from \cite{FJLS2018} that we build upon here are only able to obtain an upper bound of $2^{-\tilde\Omega({\sqrt{n}})}$ on the singularity probability. 
\end{remark}
Apart from providing a stronger conclusion, our proof of the above theorem is considerably shorter than previous works, and introduces and extends several novel combinatorial tools and ideas in discrete random matrix theory (some of which are based on joint work of the authors with Luh and Samotij \cite{FJLS2018}). We believe that these ideas allow for a unified approach to the singularity problem for many different discrete random matrix models, which have previously been handled in an ad-hoc manner (see also the discussion at the end of the next subsection). 
  
\subsection{Outline of the proof and comparison with previous work}
In this subsection, we provide a very brief, and rather imprecise, outline of our proof, and compare it to previous works of Nguyen \cite{nguyen2012inverse} and Vershynin \cite{vershynin2014invertibility}; for further comparison with the work of Costello, Tao, and Vu, see \cite{nguyen2012inverse}. 

Let $\boldsymbol{x}:=(x_1,\ldots,x_n)$ be the first row of $M_n$, let $M^{1}_{n-1}$ denote the bottom-right $(n-1)\times (n-1)$ submatrix of $M_n$, and for $2\leq i,j \leq n$, let $c_{ij}$ denote the cofactor of $M^{1}_{n-1}$ obtained by removing its $(i-1)^{st}$ row and $(j-1)^{st}$ column. Then, Laplace's formula for the determinant gives 
$$\det(M_n)=x_1\det(M_{n-1})-\sum_{i,j=2}^n c_{ij}x_ix_j,$$
so that our goal is to bound the probability (over the randomness of $x$ and $c_{ij}$) that this polynomial is zero. By a standard reduction due to \cite{costello2006random} (see \cref{lem:rank reduction,lem:second reduction,corollary:remove-first-row}), we may further assume that $M^{1}_{n-1}$ has rank either $n-2$ or $n-1$. In this outline, we will only discuss the case when $M^{1}_{n-1}$ has rank $n-1$; the other case is easier, and is handled exactly as in \cite{nguyen2012inverse} (see \cref{lemma:reduction-to-linear,eqn:conclusion-degenerate-case}).    

A decoupling argument due to \cite{costello2006random} (see \cref{lemma:decoupling-CTV}) further reduces the problem (albeit in a manner incurring a loss) to bounding from above the probability that 
$$\sum_{i\in U_1}\sum_{j \in U_2}c_{ij}(x_i - x_i')(x_j - x_j')=0,$$
where $U_1 \sqcup U_2 $ is an arbitrary non-trivial partition of $[n-1]$, and $x_i', x_j'$ are independent copies of $x_i, x_j$ (see \cref{corollary:decoupling-conclusion}). For the remainder of this discussion, the reader should think of $|U_2|$ as `small'(more precisely, $|U_2| \sim n^{1/4}\sqrt{\log{n}}$). 
We remark that a similar decoupling based reduction is used in \cite{vershynin2014invertibility} as well, whereas \cite{nguyen2012inverse} also uses a similar decoupling inequality in proving the so-called quadratic inverse Littlewood-Offord theorem. The advantage of decoupling is that for any given realization of the variables $(c_{ij})_{2\leq i,j \leq n}$ and $(x_j - x_j')_{j\in U_2}$, the problem reduces to bounding from above the probability that the \emph{linear sum} 
$$\sum_{i\in U_1}R_i(x_i - x_i')=0,$$
where $R_i:= \sum_{j \in U_2}c_{ij}(x_j - x_j')$. Problems of this form are precisely the subject of standard (linear) Littlewood-Offord theory. 

Broadly speaking, Littlewood-Offord theory applied to our problem says that the less `additive structure' the $|U_1|$-dimensional vector $(R_i)_{i\in U_1}$ possesses, the smaller the probability of the above sum being zero. Quantifying this in the form of `Littlewood-Offord type theorems' has been the subject of considerable research over the years; we refer the reader to \cite{nguyen2013small, rudelson2010non} for general surveys on the Littlewood-Offord problem with a view towards random matrix theory. Hence, our goal is to show that with very high probability, the vector $(R_i)_{i\in U_1}$ is additively `very unstructured'. This is the content of our structural theorem (\cref{thm:structural}), which is at the heart of our proof. 

The statement (and usefulness) of our structural theorem is based on the following simple, yet powerful, observations. 
\begin{itemize}
\item The $(n-1)$-dimensional vector $\boldsymbol{R}:= (R_2,\dots,R_{n})$, where recall that $R_i = \sum_{j\in U_2}c_{ij}(x_j-x_j')$, is zero if and only if $x_j = x_j'$ for all $j \in |U_2|$, which happens with probability exponentially small in $|U_2|$; the if and only if statement holds since the matrix $(c_{ij})_{2\leq i,j \leq n}$ is proportional to the matrix $(M^{1}_{n-1})^{-1}$, which is assumed to be invertible.
\item The vector $\boldsymbol{R}$ is orthogonal to at least $n-1-|U_2|$ rows of $M^{1}_{n-1}$ (\cref{lemma:R-orthogonal}). This follows since for any $2\leq j_0 \leq n$, the $n-1$ dimensional vector $(c_{ij_{0}})_{2\leq i\leq n}$ is orthogonal to all   but the $j_0^{th}$ row of $M^{1}_{n-1}$, again since the matrix $(c_{ij})_{2\leq ij \leq n}$ is proportional to the matrix $(M^{1}_{n-1})^{-1}$.  
\item The probability of the linear sum $\sum_{i \in U_1} R_i(x_i - x_i')$ being zero is `not much more' than the probability of the linear sum $\sum_{2\leq i \leq n} R_i (x_i - x_i')$ being zero (\cref{lemma:restriction-atom-prob}). 
\end{itemize}
Taken together, these observations show that it suffices to prove a structural theorem of the following form: \emph{every} non-zero integer vector which is orthogonal to `most' rows of $M^{1}_{n-1}$ is `very unstructured'. In \cite{nguyen2012inverse}, a structural theorem along similar lines is also proven. However, it suffers from two drawbacks. First, the notion of `very unstructured' in the conclusion there is much weaker, leading to the bound $O_C(n^{-C})$ for any constant $C>0$, as opposed to our bound from \cref{thm:main-thm}. Second, such a conclusion is not obtained for every non-zero integer vector, but only for those non-zero integer vectors for which `most' coefficients satisfy the additional additive constraint of being contained in a `small' generalized arithmetic progression (GAP) of `low complexity'. Consequently, the simple observations mentioned above no longer suffice, and the rest of the proof in \cite{nguyen2012inverse} is necessarily more complicated. 

The structural theorem in \cite{vershynin2014invertibility} is perhaps closer in spirit to ours, although there are many key differences, of which we mention here the most important one. Roughly speaking, both \cite{vershynin2014invertibility} and the present work prove the respective structural theorems by taking the union bound, over the choice of a non-zero (integer) vector which is not `very unstructured', that the matrix-vector product of $M^{1}_{n-1}$ with this vector is contained in a small prescribed set. \emph{A priori}, this union bound is over an infinite collection of vectors. In order to overcome this obstacle, \cite{rudelson2008littlewood,vershynin2014invertibility} adopts a geometric approach of grouping vectors on the unit sphere into a finite number of clusters based on Euclidean distances; using the union bound and a non-trivial estimate of the number of clusters to show that with very high probability, the matrix-vector product of $M^{1}_{n-1}$ with a representative of each cluster is `far' from the small prescribed set; and then, using estimates on the operator norm of $M^{1}_{n-1}$ to deduce a similar result for all other vectors in each cluster. Naturally, this geometric approach is very involved, and leads to additional losses at various steps (which is why \cite{vershynin2014invertibility} obtains a worse bound on $q_n$ than \cref{thm:main-thm}).

In contrast, we overcome this obstacle with a completely novel and purely combinatorial approach of clustering vectors based on the residues of their coordinates modulo a large prime, and using a combinatorial notion due to Hal\'asz \cite{halasz1977estimates} to quantify the amount of additive structure in a vector (\cref{prop:structural}). In particular, with our approach, the analogue of the problem of `bounding the covering number of sub-level sets of regularized LCD' -- which constitutes a significant portion of \cite{vershynin2014invertibility} (see Section 7.1 there), is one of the key contributions of that work, and is also a major contributor to the sub-optimality of the final result -- can be solved more efficiently and with a short double-counting argument (see \cref{thm:counting-lemma}, which is based on joint work of the authors with Luh and Samotij in \cite{FJLS2018}, and \cref{corollary:counting}). 

It is worth mentioning that \cite{vershynin2014invertibility} provides bounds not just for the probability of singularity of $M_{n}$, but also for the probability that the `least singular value' of $M_{n}$ (as well as random matrices with more general entries) is `very small'. Very recent work \cite{jain2019combinatorial, jain2019b, jain2019smoothed} of the second named author shows how to develop the combinatorial ideas introduced in \cite{FJLS2018} (which we use here) in order to obtain quantitative control on the lower tail of the least singular value for a variety of random matrix models. We anticipate that the ideas in the present work can be combined with those in \cite{jain2019combinatorial, jain2019b, jain2019smoothed} to control the lower tail of the least singular value of symmetric random matrices as well.  

The rest of this paper is organized as follows. In \cref{sec:proof-strategy}, we discuss in detail the overall proof strategy leading to the reduction to the structural theorem; in \cref{sec:structural theorem}, we state and prove our structural theorem; and in \cref{sec:proof-main-thm}, we put everything together to quickly complete our proof. \\ 

{\bf Notation: } Throughout the paper, we will omit floors and ceilings when they make no essential difference. For convenience, we will also say `let $p = x$ be a prime', to mean that $p$ is an odd prime between $x$ and $2x$; again, this makes no difference to our arguments. As is standard, we will use $[n]$ to denote the discrete interval $\{1,\dots,n\}$. All logarithms are natural unless noted otherwise.  

\section{Proof strategy: reduction to the structural theorem}
\label{sec:proof-strategy}
In this section, we discuss the strategy underlying our proof of \cref{thm:main-thm}. The key conclusions are \cref{eqn:split-into-deg-nondeg} \cref{eqn:conclusion-degenerate-case}, and  \cref{eqn:conclusion-nondegerate}, which show that it suffices to prove the structural theorem in \cref{sec:structural theorem} in order to prove \cref{thm:main-thm}. 
\subsection{Preliminary reductions}
For any $n \in \N$ and $k\in [n]$, let $\Ev_{k}(n)$ denote the event that $M_n$ has rank exactly $k$, and let $\Ev_{\leq k}(n)$ denote the event that $M_n$ has rank at most $k$. Thus, our goal is to bound the probability of $\Ev_{\leq n-1}(n)$. The next lemma, which is due to Nguyen \cite{nguyen2012inverse}, shows that it suffices to bound the probability of $\Ev_{n-1}(n)$. 
\begin{lemma}[Lemma 2.1 in \cite{nguyen2012inverse}]
\label{lem:rank reduction}
For any $\ell \in [n-2]$,
\[\Pr\left[\Ev_{\ell}(n)\right] \leq 0.1 \times \Pr\left[\Ev_{2n-\ell-2}(2n-\ell-1)\right].\]
\end{lemma}

The proof of this lemma uses the following simple observation due to Odlyzko \cite{Odl}:
\begin{observation}
\label{obs:odl}
  Let $V$ be any subspace of $\mathbb{R}^n$ of dimension at most $\ell$. Then, $|V\cap \{\pm 1\}^n|\leq 2^\ell$. 
\end{observation}

\begin{proof} [Proof of \cref{lem:rank reduction}]
It suffices to show that for any $\ell \leq n-2$, 
\begin{equation}
\label{eqn:rank-n+1}
\Pr\left[\Ev_{\ell+2}(n+1) \mid \Ev_{\ell}(n)\right]\geq 1-2^{-n+\ell}.
\end{equation}
Indeed, iterating this equation shows that
\begin{align*}
\Pr[\Ev_{2n-\ell-2}(2n-\ell-1)\mid \Ev_{\ell}(n)]
& \geq \prod_{j=1}^{n-\ell -1}\Pr\left[\Ev_{\ell + 2j}(n+j) \mid \Ev_{\ell + 2j-2}(n+j-1)\right] \\
& \geq \prod_{j=1}^{n-\ell-1}(1-2^{-n+\ell+j}) \geq 0.1,
\end{align*}
which gives the desired conclusion. 

In order to prove \cref{eqn:rank-n+1}, consider the coupling of $M_n$ and $M_{n+1}$ where $M_n$ is the top left $n\times n$ sub-matrix of $M_{n+1}$. 
Suppose $M_n$ has rank $\ell$, and let $V(M_n)$ be the ($\ell$-dimensional) subspace spanned by its rows. By \cref{obs:odl}, $|V(M_n)\cap \{\pm 1\}^{n}|\leq 2^\ell$. Therefore, the probability that the vector formed by the first $n$ coordinates of the last row of $M_{n+1}$ lies in $V(M_n)$ is at most $2^{-n+\ell}$. If this vector does not lie in $V(M_n)$, then the symmetry of the matrix also shows that the last column of $M_{n+1}$ does not lie in the span of the first $n$ columns of $M_{n+1}$, so that the rank of $M_{n+1}$ exceeds the rank of $M_{n}$ by $2$. 
\end{proof}

The following lemma, also due to Nguyen, allows us to reduce to the case where the rank of the $(n-1) \times (n-1)$ symmetric matrix obtained by removing the first row and the first column of $M_n$ is at least $n-2$.  
\begin{lemma}[Lemma 2.3 in \cite{nguyen2012inverse}]
\label{lem:second reduction}
Assume that $M_{n}$ has rank $n-1$. Then, there exists $i\in[n]$
such that the removal of the $i^{th}$ row and the $i^{th}$ column
of $M_{n}$ results in a symmetric matrix $M_{n-1}$ of rank at least
$n-2$. 
\end{lemma}
\begin{proof}
Without loss of generality, we can assume that the last $n-1$ rows of $M_n$ are independent. Therefore, the matrix $M_{n-1}$, which is obtained by removing the first row and first column of $M_n$ has rank at least $n-2$. 
\end{proof}
As a simple corollary of the above lemma, we obtain the following:
\begin{corollary}
\label{corollary:remove-first-row}
For $i \in [n]$, let $\Ev^{i}_{n-1}(n)$ denote the event that $M_n$ has rank $n-1$, and the symmetric matrix obtained by removing the $i^{th}$ row and the $i^{th}$ column of $M_n$ has rank at least $n-2$. Then,
$$\Pr\left[\Ev_{n-1}(n)\right] \leq n\Pr\left[\Ev^{1}_{n-1}(n)\right].$$
\end{corollary}
\begin{proof}
Suppose that $M_n$ has rank $n-1$. By \cref{lem:second reduction}, there exists an $i\in [n]$ for which the $(n-1)\times (n-1)$ matrix obtained by deleting the $i^{th}$ row and $i^{th}$ column has rank at least $n-2$. Moreover, by symmetry, 
$$\Pr[\Ev^i_{n-1}(n)]=\Pr[\Ev^1_{n-1}(n)] \text{ for all }i\in [n].$$ Therefore, by the union bound,
$$\Pr[\Ev_{n-1}(n)]= \Pr\left[\cup_{i=1}^{n}\Ev^{i}_{n-1}(n)\right] \leq  \sum_{i=1}^n\Pr[\Ev^i_{n-1}(n)]=n\Pr[\Ev_{n-1}^1(n)].$$ 
\end{proof}
Let $M^{1}_{n-1}$ denote the $(n-1)\times (n-1)$ symmetric matrix obtained by deleting the first row and first column of $M_n$. Let $\Dg(n-1)$ denote the `degenerate' event that $M^{1}_{n-1}$ has rank $n-2$, and let $\Ndg(n-1)$ denote the `non-degenerate' event that $M^{1}_{n-1}$ has full rank $n-1$.  
By definition, 
$$\Ev^{1}_{n-1}(n) = \left(\Ev^{1}_{n-1}(n)\cap\Dg(n-1)\right) \sqcup \left(\Ev^{1}_{n-1}(n)\cap \Ndg(n-1)\right),$$
and hence, 
\begin{equation}
\label{eqn:split-into-deg-nondeg}
\Pr\left[\Ev^{1}_{n-1}(n)\right] = \Pr\left[\Ev^{1}_{n-1}(n)\cap\Dg(n-1)\right]+ \Pr\left[\Ev^{1}_{n-1}(n)\cap\Ndg(n-1)\right].
\end{equation}
It is thus enough to bound each of the above two summands. 
\subsection{Bounding $\Pr\left[\Ev^{1}_{n-1}(n)\cap\Dg(n-1)\right]$}
\label{sec:degenerate-case}
Let $\boldsymbol{x}:= (x_1,\dots,x_n)$ denote the first row of $M_n$. It follows from Laplace's formula for the determinant that
\begin{equation}
\label{eqn:laplace-expansion}
\det(M_n)
= x_1\det\left(M^{1}_{n-1}\right) - \sum_{2\leq i,j \leq n}c_{ij}x_i x_j,
\end{equation}
where $c_{ij}$ denotes the cofactor of $M^{1}_{n-1}$ obtained by removing its $(i-1)^{st}$ row and $(j-1)^{st}$ column. In order to deal with $M_n \in \Ev^{1}_{n-1}(n)\cap\Dg(n-1)$, we use the following observation due to Nguyen (see Section 9 in \cite{nguyen2012inverse}).
\begin{lemma}
\label{lemma:reduction-to-linear}
For every $M_n \in \Ev^{1}_{n-1}(n)\cap\Dg(n-1)$, there exists some $\lambda:=\lambda\left(M^{1}_{n-1}\right) \in \Q\setminus\{0\}$ and some $\boldsymbol{a}:= \boldsymbol{a}\left(M^{1}_{n-1}\right)=(a_2,\dots,a_n)\in \Z^{n-1}\setminus\{\boldsymbol{0}\}$ such that
\begin{equation}
\label{eqn:conclusion-1}
M^{1}_{n-1} \boldsymbol{a} =\boldsymbol{0},
\end{equation}
and
\begin{equation}
\label{eqn:conclusion-2}
\det(M_n) = \lambda \left(\sum_{2\leq i \leq n} a_ix_i\right)^{2}.
\end{equation}
\end{lemma}
\begin{proof}
Let $\adj\left(M^{1}_{n-1}\right)$ denote the adjugate matrix of $M^{1}_{n-1}$; note that this is an integer-valued symmetric matrix since $M^{1}_{n-1}$ is an integer-valued symmetric matrix. 
Since $M^{1}_{n-1}$ is of rank $n-2$, its kernel is of rank $1$. Moreover, the equation 
\begin{equation}
\label{eqn: A adj(A) = det(A)I}
M^{1}_{n-1} \adj\left(M^{1}_{n-1}\right) = \det\left(M^{1}_{n-1}\right)I_{n-1}
\end{equation}
shows that every column of $\adj\left(M^{1}_{n-1}\right)$ is in the kernel of $M^{1}_{n-1}$ as $\det(M^{1}_{n-1}) = 0$ by assumption. It follows that the matrix $\adj\left(M^{1}_{n-1}\right)$ is an integer-valued symmetric matrix of rank $1$, which cannot be zero since $M^{1}_{n-1}$ is of rank $n-2$. Hence, there exists some $\lambda \in \Q\setminus \{0\}$ and a vector $\boldsymbol{a} = (a_2,\dots, a_n)^{T} \in \Z^{n-1}\setminus \{\0\}$ such that
\begin{equation}
\label{eqn:rank-1-symmetric-adjoint}
\adj\left(M^{1}_{n-1}\right) = \lambda \boldsymbol{a} \boldsymbol{a}^{T}.
\end{equation}
In particular, every column of $\adj\left(M^{1}_{n-1}\right)$ is equal to a multiple of the vector $\boldsymbol{a}$. By considering any column which is a non-zero multiple of $\boldsymbol{a}$, \cref{eqn: A adj(A) = det(A)I} along with $\det\left(M^{1}_{n-1}\right) = 0$ gives \cref{eqn:conclusion-1}.
Moreover, by writing the entries of the adjugate matrix in terms of the cofactors, we see that \cref{eqn:rank-1-symmetric-adjoint} is equivalent to the following: for all $2 \leq i,j \leq n$:
\begin{equation*}
c_{ij} =\lambda a_i a_j. 
\end{equation*}
Substituting this in \cref{eqn:laplace-expansion} and using $\det\left(M^{1}_{n-1}\right)=0$ gives \cref{eqn:conclusion-2}.  
\end{proof}
Before explaining how to use \cref{lemma:reduction-to-linear}, we need the following definition. 
\begin{definition}[Atom probability] 
\label{defn:atom-prob}
Let $\ring$ be an arbitrary ring (with a unit element). For a vector $\boldsymbol{a} := (a_1,\dots,a_n) \in \ring^{n}$, we define its $\mu$-atom probability by
$$\rho^{\ring}_{\mu}(\boldsymbol{a}) := \sup_{c \in \ring}\Pr_{x_1^{\mu},\dots,x_n^{\mu}}\left[a_1 x_1^{\mu} + \dots + a_n x_n^{\mu} = c\right],$$
where the $x_i^{\mu}$'s are i.i.d. random variables taking on the value $0$ with probability $\mu$ and the values $\pm 1$, each with probability $(1-\mu)/2$. 
\end{definition}
\begin{remark}
We will often refer to the $0$-atom probability simply as the atom probability, and denote it by $\rho^{\ring}(\boldsymbol{a})$ instead of $\rho^{\ring}_{0}(\boldsymbol{a})$. Similarly, we will denote $x_i^0$ simply as $x_i$. 
\end{remark}
Although we will not need them in this subsection, we will later make use of the following two simple lemmas about the atom probability. The first lemma shows that the $\mu$-atom probability of a vector is bounded above by the $\mu$-atom probability of any of its restrictions.  
\begin{lemma}
\label{lemma:sbp-monotonicity}
Let $\boldsymbol{a} \in \ring^{n}$, and let $\boldsymbol{a}|_{U_1}$ denote the restriction of $\boldsymbol{a}$ to $U_1 \subseteq [n]$. Then, 
$$\rho^{\ring}_{\mu}\left(\boldsymbol{a}\right) \leq \rho^{\ring}_{\mu}\left(\boldsymbol{a}|_{U_1}\right).$$
\end{lemma}
\begin{proof}
Let $c^{*}:=\arg\max_{c \in \ring}\Pr_{\boldsymbol{x}^{\mu}}\left[\sum_{i\in [n]}a_{i}x^{\mu}_{i}=c\right]$. Then, 
\begin{align*}
\rho^{\ring}_{\mu}(\boldsymbol{a})
&= \Pr_{\boldsymbol{x}^{\mu}}\left[\sum_{i\in[n]}a_ix_i^{\mu}=c^{*}\right] = \Pr_{\boldsymbol{x}^{\mu}}\left[\sum_{i\in[U_1]}a_ix_i^{\mu}=c^{*}-\sum_{i\in [\overline{U_1}]}a_ix_i^{\mu}\right]\\
&= \E_{(x^{\mu}_i)_{i\in \overline{U_1}}}\left[\Pr_{(x_i^{\mu})_{i\in [U_1]}}\left[\sum_{i\in[U_1]}a_ix_i^{\mu}=c^{*}-\sum_{i\in [\overline{U_1}]}a_ix_i^{\mu}\right]\right]\\
&\leq \E_{(x^{\mu}_i)_{i\in \overline{U_1}}}\left[\rho^{\ring}_{\mu}(\boldsymbol{a}|_{U_1})\right] = \rho^{\ring}_{\mu}(\boldsymbol{a}|_{U_1}),
\end{align*}
where the third equality follows from the law of total probability, and the fourth inequality follows from the definition of $\rho^{\ring}_{\mu}(\boldsymbol{a}|_{U_1})$.
\end{proof}

The second lemma complements \cref{lemma:sbp-monotonicity}, and shows that the $\mu$-atom probability cannot increase too much if, instead of the original vector, we work with its restriction to a sufficiently large subset of coordinates. 
\begin{lemma} 
\label{lemma:restriction-atom-prob}
Let $\boldsymbol{a} \in \ring^{n}$, and let $\boldsymbol{a}|_{U_1}$ denote the restriction of $\boldsymbol{a}$ to $U_1$. Then, 
$$\rho^{\ring}_{\mu}\left(\boldsymbol{a}|_{U_1}\right) \leq \max\left\{\mu,\frac{1-\mu}{2}\right\}^{-|U_2|}\rho^{\ring}_{\mu}\left(\boldsymbol{a}\right).$$
\end{lemma}
\begin{proof} 
Let $c_{0}:=\arg\max_{c \in \ring}\Pr_{\boldsymbol{x}^{\mu}}\left[\sum_{i\in U_{1}}a_{i}x^{\mu}_{i}=c\right]$
where the $x^{\mu}_{i}$'s are as in \cref{defn:atom-prob}, and let $c_{1}:=c_{0}+\sum_{i\in U_{2}}a_{i}$.
Then, 
\begin{eqnarray*}
\Pr_{\boldsymbol{x}^{\mu}}\left[\sum_{i\in[n]}a_{i}x^{\mu}_{i}=c_{0}\right] & \geq & \Pr_{(x^{\mu}_{i})_{i\in U_{1}}}\left[\sum_{i\in U_{1}}a_{i}x^{\mu}_{i}=c_{0}\right]\prod_{j\in U_{2}}\Pr_{x^{\mu}_{j}}\left[x^{\mu}_{j}=0\right]\\
 & \geq & \rho^{\ring}_{\mu}\left(\boldsymbol{a}|_{U_{1}}\right)\mu^{|U_{2}|},
\end{eqnarray*}
and 
\begin{eqnarray*}
\Pr_{\boldsymbol{x}^{\mu}}\left[\sum_{i\in[n]}a_{i}x^{\mu}_{i}=c_{1}\right] & \geq & \Pr_{(x^{\mu}_{i})_{i\in U_{1}}}\left[\sum_{i\in U_{1}}a_{i}x^{\mu}_{i}=c_{0}\right]\prod_{j\in U_{2}}\Pr_{x^{\mu}_{j}}\left[x^{\mu}_{j}=1\right]\\
 & \geq & \rho^{\ring}_{\mu}\left(\boldsymbol{a}|_{U_{1}}\right)\left(\frac{1-\mu}{2}\right)^{|U_{2}|}.
\end{eqnarray*}
Taking the maximum of the two expressions gives
\[
\rho^{\ring}_{\mu}(\boldsymbol{a})\geq\max\left\{ \mu,\frac{1-\mu}{2}\right\} ^{|U_{2}|}\rho^{\ring}_{\mu}\left(\boldsymbol{a}|_{U_{1}}\right),
\]
and by rearranging we obtain the desired conclusion. 
\end{proof}

Returning to the goal of this subsection, for $0 < \rho \leq 1$, let $\Gv_{\rho}(n-1)$ denote the event -- depending only on $M^{1}_{n-1}$ -- that \emph{every} non-zero integer null vector of $M^{1}_{n-1}$ has atom probability (in $\Z$) at most $\rho$. Then, we have   
\begin{align}
\label{eqn:conclusion-degenerate-case}
\Pr_{M_n}\left[\Ev^{1}_{n-1}(n) \cap \Dg(n-1)\right]
&\leq \Pr_{M_n}\left[\Ev^{1}_{n-1}(n) \cap \Dg(n-1) \cap \Gv_\rho(n-1) \right] + \Pr_{M^{1}_{n-1}}\left[\overline{\Gv_{\rho}(n-1)}\right] \nonumber \\
&\leq \Pr_{M^{1}_{n-1},\boldsymbol{x}}\left[\left(\sum_{2\leq i \leq n}a_i\left(M^{1}_{n-1}\right)x_i = 0\right) \cap \Gv_{\rho}(n-1) \right] + \Pr_{M^{1}_{n-1}}\left[\overline{\Gv_{\rho}(n-1)}\right] \nonumber \\ 
\leq \sum_{A_{n-1}\in \Gv_{\rho}(n-1)}&\Pr_{\boldsymbol{x}}\left[\left(\sum_{2\leq i \leq n}a_i\left(A_{n-1}\right)x_i = 0\right)\right]\Pr_{M^{1}_{n-1}}\left[M^{1}_{n-1} = A_{n-1}\right] + \Pr_{M^1_{n-1}}\left[\overline{\Gv_{\rho}(n-1)}\right]\nonumber \\
&\leq \rho + \Pr_{M^{1}_{n-1}}\left[\overline{\Gv_\rho(n- 1)}\right],
\end{align}
where the second line follows from \cref{eqn:conclusion-2}; the third line is trivial; and the last line follows from the definition of $\Gv_\rho(n-1)$. \cref{thm:structural} shows that `typically', every non-zero integer null vector of $M^{1}_{n-1}$ has `small' atom probability, and will be used to bound the right hand side of \cref{eqn:conclusion-degenerate-case}.  

\subsection{Bounding $\Pr\left[\Ev^{1}_{n-1}(n)\cap\Ndg(n-1)\right]$}
Once again, we start with \cref{eqn:laplace-expansion}. However, for $M_{n-1} \in \Ndg(n-1)$, $\adj\left(M^{1}_{n-1}\right)$ is invertible, and we no longer have the factorization of the determinant in \cref{lemma:reduction-to-linear} available to us. In this case, in order to reduce to a problem involving the anti-concentration of a linear form, we will follow an idea by Costello, Tao and Vu \cite{costello2006random}. The basic tool is the following decoupling inequality from \cite{costello2006random}.

\begin{lemma}[Lemma 4.7 in \cite{costello2006random}]
\label{lemma:decoupling-CTV}
Let $Y$ and $Z$ be independent random variables, and $E=E(Y,Z)$ be an event depending on $Y$ and $Z$. Then,
\begin{equation*}
\Pr[E(Y,Z)]^{4} \leq \Pr[E(Y,Z) \cap E(Y',Z) \cap E(Y,Z') \cap E(Y',Z')], 
\end{equation*}
where $Y'$ and $Z'$ denote independent copies of $Y$ and $Z$, respectively. 
\end{lemma}
Next, we explain how to use the above decoupling lemma for our purpose. For this discussion, recall \cref{eqn:laplace-expansion}. Fix a non-trivial partition $[n] = U_1 \sqcup U_2$. Let $Y:= (x_i)_{i \in U_1} $ and $Z:= (x_i)_{i \in U_2}$. Let $E_{\alpha, \textbf{c}} := E_{\alpha, \textbf{c}}(Y,Z)$ denote the event that
$$Q_{\alpha, \textbf{c}}(Y,Z):=\alpha - \sum_{2 \leq i,j\leq n}c_{ij}x_ix_j = 0,$$
where $\alpha$ and $\textbf{c}:=(c_{ij})_{2\leq i,j \leq n}$ are fixed. Then, the previous lemma shows that
$$\Pr\left[E_{\alpha,\textbf{c}}(Y,Z)\right]^{4} \leq \Pr\left[E_{\alpha,\textbf{c}}(Y,Z) \cap E_{\alpha,\textbf{c}}(Y',Z) \cap E_{\alpha,\textbf{c}}(Y,Z') \cap E_{\alpha,\textbf{c}}(Y',Z')\right].$$
On the other hand, whenever the event on the right holds, we also have
$$Q_{\alpha, \textbf{c}}(Y,Z) - Q_{\alpha, \textbf{c}}(Y',Z) - Q_{\alpha,\textbf{c}}(Y,Z') + Q_{\alpha,\textbf{c}}(Y,Z)=0.$$
Direct computation shows that the left hand side equals
\begin{align*}
R_{\textbf{c}}
&:= \sum_{i \in U_1}\sum_{j \in U_2}c_{ij}(x_i - x'_i)(x_j' - x_j)= \sum_{i \in U_1} R_i(x_i - x'_i),
\end{align*}
where $x'_i$ denotes an independent copy of $x_i$, and $R_i$ denotes the random sum $\sum_{j \in U_2}c_{ij}(x_j' - x_j)$. To summarize, we have deduced the following.
\begin{corollary} 
\label{corollary:decoupling-conclusion}
Let  $U_1 \sqcup U_2$ be an arbitrary non-trivial partition of $[n]$. Let $\boldsymbol{w} = (w_1,\dots,w_{|U_1|})$ be the random vector with coordinates $w_i := x_i - x'_i$. Then, with notation as above, and for any $(n-1) \times (n-1)$ symmetric matrix $A_{n-1}$, we have
$$\Pr_{M_n}\left[\Ev^{1}_{n-1}(n) \big\vert M^{1}_{n-1} = A_{n-1}\right] \leq \Pr_{\boldsymbol{x},\boldsymbol{x'}}\left[\sum_{i \in U_1}R_i w_i =0 \big\vert M^{1}_{n-1} = A_{n-1}\right]^{1/4}.$$
\end{corollary}

Using this corollary, we thus see that
\begin{align}
\label{eqn:decoupling-conclusion}
\Pr_{M_{n}}\left[\Ev_{n-1}^{1}(n)\cap\Ndg(n-1)\right]^{4} & = \left(\sum_{A_{n-1}\in\Ndg(n-1)}\Pr_{M_{n}}\left[\Ev_{n-1}^{1}(n)|M_{n-1}^{1}=A_{n-1}\right]\Pr\left[M_{n-1}^{1}=A_{n-1}\right]\right)^{4} \nonumber \\
 & \leq \sum_{A_{n-1}\in\Ndg(n-1)}\Pr_{M_{n}}\left[\Ev_{n-1}^{1}(n)|M_{n-1}^{1}=A_{n-1}\right]^{4}\Pr\left[M_{n-1}^{1}=A_{n-1}\right] \nonumber \\
 & \leq \sum_{A_{n-1}\in\Ndg(n-1)}\Pr_{\boldsymbol{x},\boldsymbol{x'}}\left[\sum_{i\in U_{1}}R_{i}w_{i}=0|M_{n-1}^{1}=A_{n-1}\right]\Pr\left[M_{n-1}^{1}=A_{n-1}\right] \nonumber \\
 & = \Pr_{\boldsymbol{x},\boldsymbol{x'},M^{1}_{n-1}}\left[\left(\sum_{i\in U_{1}}R_{i}w_{i}=0\right)\cap \Ndg(n-1)\right],
\end{align}
where the second line follows from Jensen's inequality. Hence, we have reduced the problem of bounding $\Pr\left[\Ev_{n-1}^{1}(n)\cap\Ndg(n-1)\right]$ to a linear anti-concentration problem.  

In order to use \cref{eqn:decoupling-conclusion} profitably, we will rely on the following simple, but crucial, observation about the vector $\boldsymbol{R}:=(R_2,\dots,R_n) \in \Z^{n-1}$, where $R_i$ is defined as above. 
\begin{lemma}
\label{lemma:R-orthogonal}
$\boldsymbol{R}$ is orthogonal to at least $n - 1 - |U_2|$ rows of $M^{1}_{n-1}$. 
\end{lemma}
\begin{proof}
Observe that $\boldsymbol{R}$ is a linear combination of the columns of $\adj\left(M^{1}_{n-1}\right)$ corresponding to the indices in $U_2$. By \cref{eqn: A adj(A) = det(A)I}, each of these columns is orthogonal to each of the rows with indices in $[n-1] \cap U_1$; therefore, the same is true for $\boldsymbol{R}$. Since $|[n-1] \cap U_1| \geq n-1 - |U_2|$, we are done.  
\end{proof}

For $0<\delta,\gamma \leq 1$, let $\Hv_{\delta,\gamma n}(n-1)$ denote the event -- depending only on $M^{1}_{n-1}$ -- that \emph{every} integer non-zero vector which is orthogonal to at least $(1-\gamma)n$ rows of $M^{1}_{n-1}$ has $\mu$-atom probability (in $\Z$) at most $\delta$, uniformly for all $0\leq \mu \leq 1/2$. Let $U_1 \sqcup U_2$ be a partition of $[n]$ where $U_2:= [\gamma n - 1]$. Then, with the vector $\boldsymbol{R}$ defined as above, we have
\begin{align}
\label{eqn: breaking-according-to-event-H}
\Pr_{\boldsymbol{x},\boldsymbol{x'}, M^{1}_{n-1}}\left[\left(\sum_{i \in U_1}R_i w_i=0\right) \cap \Ndg(n-1)\right]
& \leq \Pr_{\boldsymbol{x},\boldsymbol{x'}, M^{1}_{n-1}}\left[\left(\sum_{i \in U_1}R_i w_i=0\right) \cap \Hv_{\delta, \gamma n}(n-1)\cap \Ndg(n-1) \right] \nonumber \\
& \hspace{.5cm} + \Pr_{ M^{1}_{n-1}}\left[\overline{\Hv_{\delta,\gamma n}(n-1)}\right]\nonumber \\
\leq \sum_{A_{n-1} \in \Hv_{\delta, \gamma n}(n-1)\cap \Ndg(n-1)}&\Pr_{\boldsymbol{w}}\left[\sum_{i\in U_1}R_i(A_{n-1})w_i=0\right]\Pr_{ M^{1}_{n-1}}\left[M^{1}_{n-1} = A_{n-1}\right] \nonumber \\
& \hspace{.5cm}+ \Pr_{ M^{1}_{n-1}}\left[\overline{\Hv_{\delta,\gamma n}(n-1)}\right]. 
\end{align}
As in \cref{sec:degenerate-case}, we will provide an upper bound on $\Pr_{\boldsymbol{w}}\left[\sum_{i\in U_1}R_i(A_{n-1})w_i=0\right]$ which is uniform in the choice of $A_{n-1} \in \Hv_{\delta,\gamma n}(n-1)\cap \Ndg(n-1)$. We start by observing that
\begin{align}
\label{eqn:linear-independendence-consequence}
\Pr_{\boldsymbol{w}}\left[\sum_{i\in U_1}R_i(A_{n-1})w_i=0\right]
& \leq \Pr_{\boldsymbol{w}}\left[\left(\sum_{i\in U_1}R_i(A_{n-1})w_i=0\right) \cap \left(\boldsymbol{R}(A_{n-1})\neq \boldsymbol{0}\right) \right] + \Pr_{\boldsymbol{w}}\left[\boldsymbol{R}(A_{n-1})=\boldsymbol{0}\right] \nonumber \\
&= \Pr_{\boldsymbol{w}}\left[\left(\sum_{i\in U_1}R_i(A_{n-1})w_i=0\right) \cap \left(\boldsymbol{R}(A_{n-1})\neq \boldsymbol{0}\right) \right] + 2^{-|U_2|} \nonumber \\
&\leq \Pr_{\boldsymbol{w}}\left[\left(\sum_{i\in U_1}R_i(A_{n-1})w_i=0\right) \cap \left(\boldsymbol{R}(A_{n-1})\neq \boldsymbol{0}\right) \right] + 2^{-\gamma n+1}.
\end{align}
To see why the second equality holds, observe as before that $$\boldsymbol{R}(A_{n-1}):= \sum_{j \in U_2}w_j\textbf{col}_j\left(\adj\left(M^{1}_{n-1}\right)\right),$$ 
where $\textbf{col}_j\left(\adj\left(M^{1}_{n-1}\right)\right)$ denotes the $j^{th}$ column of $\adj\left(M^{1}_{n-1}\right)$. Since $A_{n-1} \in \Ndg(n-1)$, it follows that these columns are linearly independent, and hence $\boldsymbol{R}(A_{n-1}) = \boldsymbol{0}$ if and only if $w_j = 0$ for all $j \in |U_2|$, which happens precisely with probability $2^{-|U_2|}$.   

It remains to bound the first summand in \cref{eqn:linear-independendence-consequence}. For this, note that since $A_{n-1} \in \Hv_{\delta,\gamma n}(n-1)$ and $|U_2| = \gamma n -1$, \cref{lemma:R-orthogonal}, together with $\boldsymbol{R}(A_{n-1})\neq \boldsymbol{0}$, shows that $\rho^{\Z}_{1/2}\left(\boldsymbol{R}(A_{n-1})\right) \leq \delta$. Then, by \cref{lemma:restriction-atom-prob}, it follows that $\rho^{\Z}_{1/2}\left(\boldsymbol{R}(A_{n-1})|_{U_1}\right) \leq 2^{|U_2|}\delta \leq 2^{\gamma n}\delta$. Finally, combining this with \cref{eqn:decoupling-conclusion} and \cref{eqn: breaking-according-to-event-H}, we have
\begin{equation}
\label{eqn:conclusion-nondegerate}
\Pr_{M_{n}}\left[\Ev_{n-1}^{1}(n)\cap\Ndg(n-1)\right] \leq \left(2^{\gamma n}\delta + 2^{-\gamma n + 1} + \Pr_{M^{1}_{n-1}}\left[\overline{\Hv_{\delta,\gamma n}(n-1)}\right] \right)^{\frac{1}{4}}. 
\end{equation}
\section{The structural theorem}
\label{sec:structural theorem}
This section is devoted to the proof of our structural theorem, which is motivated by \cref{eqn:conclusion-degenerate-case,eqn:conclusion-nondegerate}. 
\subsection{Statement and initial reductions}
In order to state the structural theorem, we need the following definition. 
\begin{definition}
For $0\leq \alpha := \alpha(n), \beta := \beta(n) \leq 1$, let $\Hv_{\alpha, \beta n}(n)$ denote the event that every integer non-zero vector which is orthogonal to at least $(1-\beta)n$ many rows of $M_{n}$ has $\mu$-atom probability (in $\Z$) at most $\alpha$, uniformly for all $0 \leq \mu \leq 1/2$. 
\end{definition}

\begin{theorem}
\label{thm:structural}
Let $\alpha(n) = 2^{-n^{1/4}\sqrt{\log{n}}/64}$, $\beta(n) = n^{-3/4}\sqrt{\log{n}}/128$, and $n \in \N$ be sufficiently large. Then, 
$$\Pr_{M_n}\left[\overline{\Hv_{\alpha,\beta n}(n)}\right] \leq 2^{-n/32}.$$
\end{theorem}
Roughly, we will prove \cref{thm:structural} by taking a union bound, over the choice of the non-zero integer vector with large $\mu$-atom probability, of the probability that this vector is orthogonal to at least $(1-\beta)n$ many rows of $M_n$. However, there is an obstacle since, \emph{a priori}, this union bound is over an infinite collection of vectors. In order to overcome this, we will work instead with the coordinate-wise residues of the vector modulo a suitably chosen prime $p(n)$. 

In the next proposition, we make use of the event $\Hv^{p}_{\alpha, \beta n}(n)$, which is defined exactly as $\Hv_{\alpha, \beta n}(n)$, except that we work over $\F_p$ instead of the integers. 
\begin{proposition}
\label{prop:structural}
Let $\alpha(n) = 2^{-n^{1/4}\sqrt{\log{n}}/64}$ and $\beta(n) = n^{-3/4}\sqrt{\log{n}}/128$. Let $p(n) = 2^{n^{1/4}\sqrt{\log{n}}/32}$ be a prime, and let $n \in \N$ be sufficiently large. Then, 
$$\Pr_{M_n}\left[\overline{\Hv^{p}_{\alpha,\beta n}(n)}\right] \leq 2^{-n/32} .$$
\end{proposition}
Before proving \cref{prop:structural}, let us quickly show how to deduce \cref{thm:structural} from it. 
\begin{proof}[Proof of \cref{thm:structural} given \cref{prop:structural}]
It suffices to show that $\overline{\Hv_{\alpha,\beta n}(n)} \subseteq \overline{\Hv^{p}_{\alpha, \beta n}(n)}$ for any prime $p$. To see this, suppose $M_n \in \overline{\Hv_{\alpha,\beta n}(n)}$. So, there exists an integer non-zero vector $\boldsymbol{a}$ which is orthogonal to at least $(1-\beta)n$ many rows of $M_n$ and has $\mu$-atom probability (in $\Z$) greater than $\alpha$, for some $0\leq \mu \leq 1/2$. Furthermore, by rescaling $\boldsymbol{a}$ if necessary, we may assume that $\text{gcd}(a_1,\ldots,a_n)=1$. Therefore, letting $\boldsymbol{a}_{p}$ be the image of $\boldsymbol{a}$ under the natural map from $\Z^{n} \to \F_{p}^{n}$, we see that $\boldsymbol{a}_{p} \in \F_{p}^{n} \setminus \{\0\}$  and is orthogonal (over $\F_p$) to (at least) the same $(1-\beta)n$ rows of $M_n$. Finally, $\rho^{\F_p}_{\mu}(\boldsymbol{a}_p) \geq \rho^{\Z}_{\mu}(\boldsymbol{a}) > \beta$, since for any $c \in \Z$, every solution $\boldsymbol{x} \in \{-1,0,1\}^{n}$ of $a_1x_1 + \dots  + a_n x_n = c$ over the integers is also a solution of the same equation in $\F_p$. Thus, the vector $\boldsymbol{a}_p$ witnesses that $M_n \in \overline{\Hv^{p}_{\alpha, \beta n}(n)}$.        
\end{proof}

The next lemma is the first step towards the proof of \cref{prop:structural} and motivates the subsequent discussion. In its statement, the support of a vector $\boldsymbol{a}=(a_1,\dots,a_n) \in \F_{p}^{n}$, denoted by $\supp(\boldsymbol{a})$, refers to the set of indices $i\in [n]$ such that $\boldsymbol{a}_i \neq 0 \mod p$.  
\begin{lemma}
\label{lemma:eliminate-small-support}
Let $1 \leq d \leq n$ be an integer, and let $p$ be a prime. Let $\Sv^{p}_{\geq d, \beta n}(n)$ denote the event that every vector in $\F_{p}^{n}\setminus\{\0\}$ which is orthogonal (over $\F_{p}$) to at least $(1-\beta)n$ many rows of $M_{n}$ has support of size at least $d$. Suppose further that $\beta \leq 1/2$, $d\leq n/2$, $p^{\beta n} \leq 2^{n/2}$, $p^{d} \leq 2^{n/8}$, $H(\beta) \leq 1/4$, and $H(d/n) \leq 1/16$ (where $H(x):= -x\log_{2}(x) - (1-x)\log_{2}(1-x)$ is the binary entropy function for $x\in [0,1]$).  Then, 
$$\Pr_{M_n}\left[\overline{\Sv^{p}_{\geq d,\beta n}(n)}\right] \leq 2^{-n/16}.$$
\end{lemma}
The proof of this lemma will use the following simple, yet powerful, observation. 
\begin{observation}
\label{obs:inject-permutation}
Let $\Sigma$ be an $n\times n$ permutation matrix. Then, for a uniformly random $n\times n$ symmetric $\{\pm 1\}$-matrix $M_n$, the random matrix $\Sigma^{-1}M_n\Sigma$ is also a uniformly distributed $n\times n$ symmetric $\{\pm 1\}$-matrix.   
\end{observation}
\begin{proof}
It is clear than $\Sigma^{-1}M_n\Sigma$ is an $n\times n$ $\{\pm 1\}$-matrix. That it is symmetric follows from $\Sigma^{-1} = \Sigma^{T}$ and $M_n^{T} = M_n$. Finally, $\Sigma^{-1}M_n\Sigma$ is uniformly distributed since conjugation by $\Sigma$ is manifestly a bijection from the set of $n\times n$ $\{\pm 1\}$ symmetric matrices to itself.  
\end{proof}
\begin{proof}[Proof of \cref{lemma:eliminate-small-support}]
Let $d$ be as in the statement of the lemma, and  for $1\leq s \leq d$, let $\Supp_{=s}(n)$ denote the set of all vectors in $\F_{p}^{n}$ which have support of size exactly $s$. Observe that $|\Supp_{=s}(n)| \leq \binom{n}{s}p^{s}$. We will now bound the probability that any given $\boldsymbol{a} \in \Supp_{=s}(n)$ is orthogonal to at least $(1-\beta)n$ rows of a uniformly chosen $M_n$.  

For this, let $\Sigma = \Sigma(\boldsymbol{a})$ denote a fixed, but otherwise arbitrary, permutation matrix for which $\Sigma \1_{\supp(\boldsymbol{a})} = \1_{[n-s+1,n]}$. In other words, $\Sigma$ permutes the vector $\boldsymbol{a}$ so that its nonzero entries are placed in the last $s$ coordinates. Since \cref{obs:inject-permutation} shows that $\Sigma^{-1} M_n \Sigma $ is a uniformly random $n \times n$ $\{\pm 1\}$-symmetric matrix, it follows that 
\begin{align}
\label{eqn:prob-orthog-many}
\Pr_{M_n}[\boldsymbol{a} \text{ is orthogonal to $\geq (1-\beta)n$ rows of } M_n] 
&= \Pr_{M_n}\left[\boldsymbol{a} \text{ is orthogonal to $\geq (1-\beta)n$ rows of } \Sigma^{-1}M_n\Sigma\right]\nonumber \\
&= \Pr_{M_n}\left[\Sigma^{-1}M_n\Sigma \boldsymbol{a} = \boldsymbol{v} \text{ for some } \boldsymbol{v}\in \bigcup_{t=0}^{\beta n}{\Supp_{=t}(n)}\right] \nonumber \\
&\leq \sum_{t=0}^{\beta n} \Pr_{M_n}\left[\Sigma^{-1}M_n \Sigma \boldsymbol{a} = \boldsymbol{v} \text{ for some } \boldsymbol{v}\in \Supp_{=t}(n)\right] \nonumber \\
&= \sum_{t=0}^{\beta n}\Pr_{M_n}\left[M_n \Sigma \boldsymbol{a} = \boldsymbol{v} \text{ for some } \boldsymbol{v} \in \Supp_{=t}(n)\right] \nonumber \\
& \leq \sum_{t=0}^{\beta n} \sum_{\boldsymbol{v}\in \Supp_{=t}(n)}\Pr_{M_n}\left[M_n\Sigma \boldsymbol{a} = \boldsymbol{v}\right],
\end{align}
where the third line follows by the union bound; the fourth line follows since the size of the support of a vector is invariant under the action of $\Sigma$; and the last line follows again by the union bound.  

Next, we provide a (crude) upper bound on $\Pr_{M_n}\left[M_n(\Sigma \boldsymbol{a}) = \boldsymbol{v}\right]$ for any fixed $\boldsymbol{v}=(v_1,\dots,v_n) \in \F_{p}^{n}$. For this, we isolate the last column of the matrix $M_n$ by rewriting the system of equations $M_n (\Sigma \boldsymbol{a}) = \boldsymbol{v}$ as
\begin{equation}
\label{eqn:prob-fixed-rhs}
m_{in} = (\Sigma \boldsymbol{a})_{n}^{-1}\left(v_{i}-\sum_{j=1}^{n-1}m_{ij}(\Sigma \boldsymbol{a})_{j}\right) \text{ for all } i\in [n],
\end{equation}
where $m_{ij}$ denotes the $(i,j)^{th}$ entry of the matrix $M_n$, and the equation makes sense since $(\Sigma \boldsymbol{a})_n \neq 0$ by our choice of $\Sigma$. Note that the right hand side of the equation is completely determined by the top-left $(n-1)\times (n-1)$ submatrix of $M_n$. Further, the entries $m_{in}, i \in [n]$ are mutually independent even after conditioning on any realisation of the top-left $(n-1)\times (n-1)$ submatrix of $M_n$. Since $m_{in}$ takes on any value with probability at most $1/2$, it follows that conditioned on any realisation of the top-left $(n-1)\times (n-1)$ submatrix of $M_n$, \cref{eqn:prob-fixed-rhs} is satisfied with probability at most $(1/2)^{n}$. Hence, by the law of total probability, $\Pr_{M_n}[M_n\Sigma \boldsymbol{a} = \boldsymbol{v}]\leq 2^{-n}$. Substituting this in \cref{eqn:prob-orthog-many}, we see that 
\begin{align}
\label{eqn:easy-entropy-bound}
\Pr_{M_n}[\boldsymbol{a} \text{ is orthogonal to $\geq (1-\beta)n$ rows of } M_n] 
&\leq 2^{-n}\sum_{t=0}^{\beta n} |\Supp_{=t}(n)| \nonumber \\
&\leq 2^{-n}\sum_{t=0}^{\beta n}\binom{n}{t}p^{t} \leq 2^{-n}p^{\beta n}\sum_{t=0}^{\beta n}\binom{n}{t} \nonumber \\
& \nonumber \\
&\leq 2^{-n/2}2^{nH(\beta)} \leq 2^{-n/4},
\end{align}
where the fourth inequality follows by the assumption on $p^{\beta n}$ and the standard inequality $\sum_{t=0}^{\beta n}\binom{n}{t} \leq 2^{n H(\beta)}$ for $\beta \leq 1/2$, and the last inequality follows by the assumption on $nH(\beta)$. Finally, we have
\begin{align*}
\Pr_{M_n}\left[\overline{\Sv^{p}_{\geq d,\beta n}(n)}\right]
&\leq \sum_{s=1}^{d}\sum_{\boldsymbol{a}\in \Supp_{=s}(n)}\Pr_{M_n}[\boldsymbol{a} \text{ is orthogonal to $\geq (1-\beta)n$ rows of } M_n] \\
&\leq 2^{-n/4}\sum_{s=1}^{d}|\Supp_{=s}(n)| \leq 2^{-n/4}\sum_{s=1}^{d}\binom{n}{s}p^{s}\\
&\leq 2^{-n/4}p^{d}\sum_{s=1}^{d}\binom{n}{s} \leq 2^{-n/8}2^{nH(d/n)} \leq 2^{-n/16},
\end{align*}
where the fifth inequality follows by the assumption on $p^{d}$ and $d$, and the last inequality follows by the assumption on $H(d/n)$. 

\end{proof}
\subsection{Tools and auxiliary results}
Following \cref{lemma:eliminate-small-support}, we will bound $\Pr_{M_n}\left[\overline{\Hv^{p}_{\alpha,\beta n}(n)}\cap \Sv^{p}_{\geq d,\beta n}(n)\right]$ for suitably chosen parameters. Our proof of this bound will be based on the following two key ingredients. The first is a classical anti-concentration inequality due to Hal\'asz, which bounds the atom probability of a vector in terms of the `arithmetic structure' of its coordinates. In order to state it, we need the following definition.
\begin{definition}
Let $\boldsymbol{a} \in \F_{p}^{n}$ and let $k \in \N$. We define $R_k(\boldsymbol{a})$ to be the number of solutions to 
$$\pm a_{i_1} \pm a_{i_2}\pm \dots \pm a_{i_{2k}} = 0 \mod p,$$
where repetitions are allowed in the choice of $i_1,\dots,i_{2k} \in [n]$. 
\end{definition}

\begin{theorem}[Hal\'asz, \cite{halasz1977estimates}]
  \label{thm:halasz}
Let $p$ be any odd prime and let $\boldsymbol{a}:=(a_1,\ldots,a_n) \in \F_p^{n}\setminus \{\0\}$.  
Then,
  $$\sup_{0\leq \mu \leq \frac{1}{2}}\max_{q\in \F_p}\Pr\left[\sum_ia_ix_i^{\mu} = q \right]\leq \frac{1}{p}+\frac{CR_k(\boldsymbol{a})}{2^{2k} n^{2k}f(|\supp(\boldsymbol{a})|)^{1/2}} + e^{-f(|\supp(\boldsymbol{a})|)/2},$$
where $C$ is an absolute constant (which we may assume is at least $1$), and $f(|\supp(\boldsymbol{a})|)$ is a positive real number which is at most $\min\{|\supp(\boldsymbol{a})|/100, n/k\}$.
\end{theorem}
Hal\'asz's inequality is typically stated and proved over the integers, but the version over $\F_p$ stated above easily follows using the same ideas. For the reader's convenience, we provide a complete proof in \cref{app:halasz}. 

The second ingredient is a `counting lemma' due to the authors together with Luh and Samotij \cite{FJLS2018}, which bounds the number of vectors in $\F_{p}^{n}$ with a slightly different (but practically equivalent) notion of `rich additive structure'. 
\begin{definition}
Let $\boldsymbol{a}\in \F_{p}^{n}$ and let $k \in \N$. We define $R_k^*(\boldsymbol{a})$ to be the number of solutions to
$$\pm a_{i_1}\pm a_{i_2}\dots \pm a_{i_{2k}}=0 \mod p$$ 
that satisfy $|\{i_1,\dots,i_{2k}\}| \geq 1.01k$. 
\end{definition}
As mentioned above, $R_k(\boldsymbol{a})$ and $R_k^*(\boldsymbol{a})$ are practically equivalent. This is made precise by the following lemma. 
\begin{lemma}[Lemma 1.6 in \cite{FJLS2018}]
\label{lemma:R_k vs R_k^*}
For all positive integers $k,n$ with $k\leq n/2$ and any vector $\boldsymbol{a} \in \F^n_p$, 
\[R_k(\boldsymbol{a})\leq  R_k^*(\boldsymbol{a}) + (40k^{0.99}n^{1.01})^k.\]
\end{lemma}
\begin{proof}
By definition, $R_k(\boldsymbol{a})$ is equal to $R_k^*(\boldsymbol{a})$ plus the number of solutions to $\pm a_{i_1}\pm a_{i_2}\pm\dots\pm a_{i_{2k}} = 0$ that satisfy $|\{i_1,\dots,i_{2k}\}| < 1.01k$. The latter quantity is bounded from above by the number of sequences $(i_1, \dotsc, i_{2k}) \in [n]^{2k}$ with at most $1.01k$ distinct entries times $2^{2k}$, the number of choices for the $\pm$ signs. Thus
  \[
    R_k(\boldsymbol{a}) \leq R_k^*(\boldsymbol{a}) + \binom{n}{1.01k} \big(1.01k\big)^{2k}2^{2k} \leq R_k^*(\boldsymbol{a}) +  \left(4e^{1.01}k^{0.99}n^{1.01}\right)^k,
  \]
  where the final inequality follows from the well-known bound $\binom{a}{b} \le (ea/b)^b$. Finally, noting that $4e^{1.01} \leq 40$ completes the proof.
\end{proof}
We can now state the `counting lemma' from \cite{FJLS2018}. In the following statement, the notation $\boldsymbol{b}\subset \boldsymbol{a}$ for $\boldsymbol{a} \in \F_{p}^{n}$ means that $\boldsymbol{b}$ is a sub-vector of $\boldsymbol{a}$ i.e. an element of $\cup_{s=1}^{n} \F_{p}^{s}$ formed by retaining some of the entries of $\boldsymbol{a}$; the dimension of $\boldsymbol{b}$ is denoted by $|\boldsymbol{b}|$.     
\begin{theorem}[Theorem 1.7 in \cite{FJLS2018}]
  \label{thm:counting-lemma}
  Let $p$ be a prime and let $k \in \N, s\in [n], t\in [p]$. 
Let 
$$\Bad_{k,s,\geq t}(n):= \left\{\boldsymbol{a} \in \F_{p}^{n} \mid \forall \boldsymbol{b}\subset \boldsymbol{a} \text{ s.t. } |\boldsymbol{b}|\geq s \text{ we have } R^*_k(\boldsymbol{b})\geq t\cdot \frac{2^{2k}\cdot
|\boldsymbol{b}|^{2k}}{p}\right\}$$
denote the set of `$k,s,\geq t$-bad vectors'.   Then, 
$$|\Bad_{k,s,\geq t}(n)| \leq \left(\frac{s}{n}\right)^{2k-1}p^{n}(0.01t)^{-n+s}.$$  
\end{theorem}
The above theorem shows that there are very few vectors for which every sufficiently large subset has rich additive structure. However, in order to use the strategy in the proof of \cref{lemma:eliminate-small-support} effectively, we require that there are very few vectors for which every \emph{moderately-sized} subset has rich additive structure  (see the proof of \cref{corollary:prob-orth-fixed-vector}). This is accomplished by the following corollary.   

\begin{corollary}
\label{corollary:counting}
Let $p$ be a prime and let $k, s_1,s_2,d\in [n], t\in [p]$ such that $s_1 \leq s_2$. 
Let 
$$\Bad^{d}_{k,s_1,s_2,\geq t}(n):= \left\{\boldsymbol{a} \in \F_{p}^{n} \big\vert |\supp(\boldsymbol{a})|=d \text{ and } \forall \boldsymbol{b}\subset \boldsymbol{a}|_{\supp(\boldsymbol{a})} \text{ s.t. } s_2\geq|\boldsymbol{b}|\geq s_1 : R^*_k(\boldsymbol{b})\geq t\cdot \frac{2^{2k}\cdot |\boldsymbol{b}|^{2k}}{p}\right\}.$$
Then, 
$$|\Bad^{d}_{k,s_1,s_2,\geq t}(n)| \leq {n \choose d}p^{d+s_{2}}(0.01t)^{-d+\frac{s_{1}}{s_{2}}d}.$$  
\end{corollary}
\begin{proof}
At the expense of an overall factor of $\binom{n}{d}$, we may restrict our attention to those vectors in $\Bad^{d}_{k,s_1,s_2,\geq t}(n)$ whose support is $[d]$. In order to count the number of such vectors, we begin by decomposing $[d]$ into the intervals $I_1,\dots,I_{m+1}$, where $m:=\lfloor d/s_2 \rfloor$, $I_j:= \{(j-1)s_2+1,\dots,js_2\}$ for $j\in [m]$, and $I_{m+1}:= \{ms_2+1,\dots,d\}$. For a vector with support $[d]$ to be in $\Bad^{d}_{k,s_1,s_2,\geq t}(n)$, it must necessarily be the case that the restriction of the vector to each of the intervals $I_1,\dots,I_m$ is in $\Bad_{k,s_1,\geq t}(s_2)$. Since there are at most $p^{|I_{m+1}|}\leq p^{s_2}$ many choices for the restriction of the vector to $I_{m+1}$, it follows from \cref{thm:counting-lemma} that
\begin{eqnarray*}
|\Bad_{k,s_{1},s_{2},\geq t}^{d}(n)| & \leq & {n \choose d}\left|\Bad_{k,s_{1},\geq t}(s_{2})\right|^{m}p^{s_{2}}
  \leq {n \choose d}\left\{ \left(\frac{s_{1}}{s_{2}}\right)^{2k-1}p^{s_{2}}(0.01t)^{-s_{2}+s_{1}}\right\} ^{m}p^{s_{2}}\\
 & \leq & {n \choose d}\left(p^{s_{2}}(0.01t)^{-s_{2}+s_{1}}\right)^{\frac{d}{s_{2}}}p^{s_{2}}= {n \choose d}p^{d+s_{2}}(0.01t)^{-d+\frac{s_{1}}{s_{2}}d}.
\end{eqnarray*}
\end{proof}

We conclude this subsection with a few corollaries of \cref{thm:halasz} and \cref{corollary:counting}. Let $\boldsymbol{a} \in \Supp_{=d}(n) \setminus \Bad^{d}_{k,s_1,s_2,\geq (t+1)}(n)$ for $s_1\leq d\leq n$. Then, by definition, there exists $\Lambda=\Lambda(\boldsymbol{a}) \subseteq \supp(\boldsymbol{a})$ such that $s_1 \leq |\Lambda |=|\supp(\boldsymbol{a}|_\Lambda)|\leq s_2$ and $R_k^*(\boldsymbol{a}|_\Lambda)< (t+1)\cdot 2^{2k}|\Lambda|^{2k}/p$. From now on, fix such a subset $\Lambda(\boldsymbol{a})$ for every such vector $\boldsymbol{a}$.   

\begin{corollary}
\label{corollary:halasz-usable}
Let $p$ be a prime and let $\boldsymbol{a} \in \Supp_{=d}(n) \setminus \Bad^{d}_{k,s_1,s_2,\geq (t+1)}(n)$ for $1\leq s_1 \leq d \leq n$. Suppose $p^{-1} \geq \max\left\{e^{-s_1/2k}, \left(50k/s_1\right)^{0.99k}\right\}$ and $t \geq s_1 \geq k \geq 100$. Then, 
$$\sup_{0\leq \mu \leq \frac{1}{2}}\rho_{\mu}^{\F_p}(\boldsymbol{a}|_{\Lambda(\boldsymbol{a})})\leq \frac{2Ct\sqrt{k}}{p\sqrt{s_1}},$$
where $C \geq 1$ is an absolute constant.
\end{corollary}
\begin{proof}
For convenience of notation, let $\boldsymbol{b}:= \boldsymbol{a}|_{\Lambda(\boldsymbol{a})}$. 
By applying \cref{thm:halasz} to the vector $\boldsymbol{b}$ with $f(|\supp(\boldsymbol{b})|) := |\supp(\boldsymbol{b})|/k = |\boldsymbol{b}|/k =: f(|\boldsymbol{b}|)$ (which is a valid choice for $f$ since $k\geq 100$ by assumption), we get 
\begin{eqnarray*}
\sup_{0\leq\mu\leq\frac{1}{2}}\rho_{\mu}^{\F_p}(\boldsymbol{b}) 
& \leq & \frac{1}{p}+\frac{C\left(R_{k}^{*}(\boldsymbol{b})+(40k^{0.99}|\boldsymbol{b}|^{1.01})^{k}\right)}{2^{2k}|\boldsymbol{b}|^{2k}\sqrt{|\boldsymbol{b}|/k}}+e^{-|\boldsymbol{b}|/2k}\\
 & \leq & \frac{1}{p}+\frac{C(t+1)}{p\sqrt{|\boldsymbol{b}|/k}}+\frac{C(40k^{0.99})^{k}}{|\boldsymbol{b}|^{0.99k}\sqrt{|\boldsymbol{b}|/k}}+e^{-|\boldsymbol{b}|/2k}\\
 & \leq & \frac{1}{p}+\frac{C(t+1)\sqrt{k}}{p\sqrt{|\boldsymbol{b}|}}+\frac{C(40k^{0.99})^{k}}{|\boldsymbol{b}|^{0.99k}}+e^{-|\boldsymbol{b}|/2k}\\
 & \leq & \frac{1}{p}+\frac{C(t+1)\sqrt{k}}{p\sqrt{s_{1}}}+C\left(\frac{50k}{s_{1}}\right)^{0.99k}+e^{-s_{1}/2k}\\
 & \leq & \frac{(2+C)}{p}+\frac{C(t+1)\sqrt{k}}{p\sqrt{s_{1}}} \leq  \frac{2Ct\sqrt{k}}{p\sqrt{s_{1}}},
\end{eqnarray*}
where the first line follows from \cref{thm:halasz}, \cref{lemma:R_k vs R_k^*}, and the choice of $\Lambda(\boldsymbol{a})$, the fifth line follows by the assumption on $p$, and the last line follows since $t \geq s_1 \geq 100$.  
\end{proof}

\begin{corollary}
\label{corollary:prob-orth-fixed-vector}
Let $p$ be a prime and let $\boldsymbol{a} \in \Bad^{d}_{k,s_1,s_2,\geq t}(n) \setminus \Bad^{d}_{k,s_1,s_2,\geq (t+1)}(n)$. Suppose $p^{-1} \geq \max\left\{e^{-s_1/2k}, \left(50k/s_1\right)^{0.99k}\right\}$, $n\geq d\geq s_1$, and $t \geq s_1 \geq k \geq 100$. Then, for $0\leq \beta := \beta(n) \leq 1/2$,
$$\Pr_{M_n}[\boldsymbol{a} \text{ is orthogonal to $\geq (1-\beta)n$ rows of } M_n] \leq 2^{n H(\beta)}p^{\beta n}\left(\frac{2Ct\sqrt{k}}{p\sqrt{s_1}}\right)^{n-s_2},$$
where $C\geq 1$ is an absolute constant. 
\end{corollary}
\begin{proof}
The proof is very similar to the proof of \cref{lemma:eliminate-small-support}. 
Let $\Lambda:=\Lambda(\boldsymbol{a})$ and $\boldsymbol{b}:= \boldsymbol{a}|_{\Lambda}$. As in the proof of \cref{lemma:eliminate-small-support}, let $\Sigma$ denote a fixed, but otherwise arbitrary, permutation matrix for which $\Sigma \1_{\Lambda} = \1_{[n-|\Lambda|+1,n]}$. Then, by \cref{eqn:prob-orthog-many},
\begin{align*}
\Pr_{M_n}[\boldsymbol{a} \text{ is orthogonal to $\geq (1-\beta)n$ rows of } M_n] 
= \sum_{t=0}^{\beta n}\sum_{\boldsymbol{v} \in \Supp_{=t}}\Pr_{M_n}\left[M_n \Sigma \boldsymbol{a} = \boldsymbol{v}\right].
\end{align*}
Next, we provide an upper bound on $\Pr_{M_n}\left[M_n(\Sigma \boldsymbol{a})=\boldsymbol{v}\right]$ for any fixed $\boldsymbol{v}=(v_1,\dots,v_n)\in \F_{p}^{n}$. For this, note that the system of equations $M_n(\Sigma \boldsymbol{a})=\boldsymbol{v}$ implies in particular that
\begin{equation}
\label{eqn:orthogonal-many-2}
\sum_{j=1}^{|\Lambda|}m_{i,n-|\Lambda|+j}b_j = v_i - \sum_{j=1}^{n-|\Lambda|}m_{i,j}(\Lambda \boldsymbol{a})_j \text{ for all }i\in [n-|\Lambda|],  
\end{equation}
Note that the right hand side is completely determined by the top-left $(n-|\Lambda|)\times (n-|\Lambda|)$ submatrix of $M_n$, and the entries of $M_n$ appearing on the left are mutually independent even after conditioning on any realisation of the top-left $(n-|\Lambda|)\times (n-|\Lambda|)$ submatrix of $M_n$. In particular, after conditioning on any realisation of the top-left submatrix of this size, each of the $n-|\Lambda|$ equations above is satisfied with probability which is at most $\rho^{\F_p}(\boldsymbol{b})$, and the satisfaction of different equations is mutually independent. Hence, by the law of total probability, the system \cref{eqn:orthogonal-many-2} is satisfied with probability at most $$\left(\rho^{\F_p}(\boldsymbol{b})\right)^{n-|\Lambda|} \leq \left(\frac{2Ct\sqrt{k}}{p\sqrt{s_1}}\right)^{n-|\Lambda|}\leq \left(\frac{2Ct\sqrt{k}}{p\sqrt{s_1}}\right)^{n-s_2} ,$$
where the middle bound follows from \cref{corollary:halasz-usable}, and the right-hand bound follows since $|\Lambda|\leq s_2$. Finally, substituting this in \cref{eqn:prob-orthog-many} and proceeding as in \cref{eqn:easy-entropy-bound} gives the desired conclusion.  
\end{proof}

\begin{corollary}
\label{corollary:prob-level-set}
Let $p$ be a prime and  $k,s_1,s_2,d\in [n], t\in [p]$ be such that $1\leq s_1 \leq s_2 \leq n/2$, $s_1 \leq d \leq n$, $p^{-1} \geq \max\left\{e^{-s_1/2k}, \left(50k/s_1\right)^{0.99k}\right\}$, and $t \geq s_1 \geq k \geq 100$.
Then, for $0\leq \beta := \beta(n) \leq 1/2$,
\begin{align*}
\Pr_{M_n}[\exists \boldsymbol{a} \in \Bad^{d}_{k,s_1,s_2,\geq t}(n)\setminus \Bad^{d}_{k,s_1,s_2,\geq (t+1)}(n) : \boldsymbol{a} \text{ is orthogonal to $\geq (1-\beta)n$ rows of } M_n] \\ \leq (500C)^{n}p^{\beta n + 2s_2 + \frac{s_1}{s_2}d}\left(\frac{k}{s_1}\right)^{n/4},
\end{align*}
where $C\geq 1$ is an absolute constant. 
\end{corollary}
\begin{proof}
Using \cref{corollary:prob-orth-fixed-vector} to bound the probability that any given $\boldsymbol{a} \in \Bad^{d}_{k,s_1,s_2,\geq t}(n)\setminus \Bad^{d}_{k,s_1,s_2,\geq (t+1)}(n)$ is orthogonal to at least $(1-\beta)n$ rows of $M_n$, and taking the union bound over all $|\Bad^{d}_{k,s_1,s_2, \geq t}(n)\setminus \Bad^{d}_{k,s_1,s_2, \geq (t+1)}(n)|$ such vectors $\boldsymbol{a}$, we see that the desired probability is at most
\begin{align*}
|\Bad^{d}_{k,s_1,s_2, \geq t}(n)\setminus \Bad^{d}_{k,s_1,s_2, \geq (t+1)}(n)|\cdot 2^{n H(\beta)}p^{\beta n}\left(\frac{2Ct\sqrt{k}}{p\sqrt{s_1}}\right)^{n-s_2}
&\leq |\Bad^{d}_{k,s_1,s_2,\geq t}(n)|\cdot 2^{n}p^{\beta n}\left(\frac{2Ct\sqrt{k}}{p\sqrt{s_1}}\right)^{n-s_2}\\
\leq 2^{n}\binom{n}{d}p^{d+s_2}(0.01t)^{-d+\frac{s_1}{s_2}d}p^{\beta n}\left(\frac{2Ct\sqrt{k}}{p\sqrt{s_1}}\right)^{n-s_2}
\leq& (500C)^{n}p^{\beta n+s_2+\frac{s_1}{s_2}d}\left(\frac{t}{p}\right)^{n-d-s_2}\left(\frac{k}{s_1}\right)^{n/4}\\
\leq (500C)^{n}p^{\beta n + 2s_2 + \frac{s_1}{s_2}d}\left(\frac{k}{s_1}\right)^{n/4},
\end{align*}
where the second inequality follows from \cref{corollary:counting}, and the third inequality follows from $s_2 \leq n/2$. 
\end{proof}
\subsection{Proof of \cref{prop:structural}}
By combining the results of the previous subsection, we can now prove \cref{prop:structural}. 
\begin{proof}[Proof of \cref{prop:structural}]
Consider the following choice of parameters: $k = n^{1/4}$, $s_1 = n^{1/2}\log{n}$, $s_2 = n^{3/4}\sqrt{\log{n}}$, $\beta n= n^{1/4}\sqrt{\log{n}}/128$, $d = n^{2/3}$, $\alpha = 2^{-n^{1/4}\sqrt{\log{n}}/64}$, and $p = 2^{n^{1/4}\sqrt{\log{n}}/32}$. Throughout, we will assume that $n$ is sufficiently large for various inequalities to hold, even if we do not explicitly mention this. \\

{\bf Step 1: }It is readily seen that the assumptions of \cref{lemma:eliminate-small-support} are satisfied, so that $\Pr\left[\overline{\Sv^{p}_{\geq d, \beta n}(n)}\right] \leq 2^{-n/16}$. In other words, except with probability at most $2^{-n/16}$, every vector in $\F_{p}^{n}\setminus\{\0\}$ which is orthogonal to at least $(1-\beta)n$ rows of $M_n$ has support of size at least $d = n^{2/3}$. \\

{\bf Step 2: }Let $\boldsymbol{a} \in \Supp_{=s}(n)\setminus \Bad^{s}_{k,s_1,s_2,\geq \sqrt{p}}(n)$ for any $s\geq d$. Since the assumptions of \cref{corollary:halasz-usable} are satisfied for our choice of parameters, it follows from \cref{corollary:halasz-usable} and \cref{lemma:sbp-monotonicity} that for any $0\leq \mu \leq 1/2$, 
$$\rho_{\mu}^{\F_p}(\boldsymbol{a}) \leq \rho_{\mu}^{\F_p}(\boldsymbol{a}|_{\Lambda(\boldsymbol{a})})\leq \frac{2C\sqrt{k}}{\sqrt{p{s_1}}} \leq \alpha ,$$
for all $n$ sufficiently large.\\ 

{\bf Step 3: }Therefore, it suffices to bound the probability that for some $s\geq d$, there exists some vector in $\Bad^{s}_{k,s_1,s_2,\geq \sqrt{p}}(n)$ which is orthogonal to at least $(1-\beta)n$ rows of $M_n$. By writing 
$$\Bad^{s}_{k,s_1,s_2,\geq \sqrt{p}}(n):= \bigcup_{t=\sqrt{p}}^{p}\Bad^{s}_{k,s_1,s_2,\geq t}(n)\setminus \Bad^{s}_{k,s_1,s_2,\geq (t+1)}(n),$$
noting that the assumptions of \cref{corollary:prob-level-set} are satisfied, and taking the union bound over the choice of $s$ and $t$, it follows that this event has probability at most 
\begin{align*}
np(500C)^{n}p^{\beta n + 2s_2 + \frac{s_1}{s_2}s}\left(\frac{k}{s_1}\right)^{n/4}
&\leq np(500C)^{n}p^{4s_2}2^{-(n\log{n})/16}\\
&\leq np(500C)^{n}2^{-(n\log{n})/32} \leq 2^{-(n\log{n})/64},
\end{align*}
for all $n$ sufficiently large.\\

Combining these steps, it follows that
$$\Pr_{M_n}\left[\overline{\Hv^{p}_{\alpha, \beta n}}\right] \leq 2^{-n/16} + 2^{-(n\log{n})/64} \leq 2^{-n/32},$$
as desired. 
\end{proof}

\section{Proof of \cref{thm:main-thm}}
\label{sec:proof-main-thm}
Our main result is now immediate. 
\begin{proof}[Proof of \cref{thm:main-thm}]
By definition, $\overline{\Gv_{\rho}(n-1)} \subseteq \overline{\Hv_{\rho, \beta n}(n-1)}$ for every $\beta \geq 0$. Therefore, from \cref{eqn:split-into-deg-nondeg}, \cref{eqn:conclusion-degenerate-case}, and \cref{eqn:conclusion-nondegerate}, it follows that
\begin{align*}
\Pr_{M_n}\left[\Ev^{1}_{n-1}\right] \leq \alpha + \Pr_{M^{1}_{n-1}}\left[\overline{\Hv_{\alpha, \beta n}(n-1)}\right] + \left(2^{\beta n}\alpha + 2^{-\beta n + 1} + \Pr_{M^{1}_{n-1}}\left[\overline{\Hv_{\rho, \beta n}(n-1)}\right] \right)^{1/4},
\end{align*}
where $\alpha$ and $\beta$ are as in the statement of \cref{thm:structural}. From \cref{thm:structural}, it follows that the right hand side of the above equation is at most $2^{-n^{1/4}\sqrt{\log{n}}/600}$ for all $n$ sufficiently large. Finally, \cref{lem:rank reduction} and \cref{corollary:remove-first-row} give the desired conclusion. 
\end{proof}
\bibliographystyle{abbrv}
\bibliography{symmetric}
\appendix
\section{Proof of Hal\'asz's inequality over $\F_{p}$}
\label{app:halasz}
In this appendix, we prove \cref{thm:halasz}. The proof follows Hal\'asz's original proof in \cite{halasz1977estimates}. 
\begin{proof}[Proof of \cref{thm:halasz}]
  Let $e_p$ be the canonical generator of the Pontryagin dual of $\F_p$, that is, the function $e_p \colon \F_p \to \C$ defined by $e_p(x) = \exp(2\pi i x / p)$. Recall the following discrete Fourier identity in $\F_p$:
  \[
    \delta_0(x) = \frac{1}{p} \sum_{r \in \F_p}e_p(rx),
  \]
  where $\delta_0(0) = 1$ and $\delta_0(x) = 0$ if $x \neq 0$. 
 Note that for any $q \in \F_p$, 
\begin{align*}
\Pr_{x^{\mu}}\left[\sum_{i=1}^{n} a_i x_i^{\mu} = q\right] &= \E_{x^{\mu}}\left[\delta_0 \left(\sum_{i=1}^{n} a_i x_i^{\mu} - q\right)\right]  \\
                          &= \E_{x^{\mu}}\left[ \frac{1}{p} \sum_{r\in \F_p} e_p\left(r \left(\sum_{j=1}^{n} a_j x_j^{\mu} - q\right)\right)\right]  \\
                          &= \E_{x^{\mu}}\left[\frac{1}{p} \sum_{r \in \F_p} \prod_{j=1}^{n} e_p\left( r a_j x_j^{\mu} \right) e_p(-r q)\right] \\
                          &\leq \frac{1}{p} \sum_{r\in \F_p} \prod_{j=1}^{n} \left| \mu + (1-\mu)\cos\left(\frac{2 \pi r a_j}{p}\right) \right| \\
                          &= \frac{1}{p} \sum_{r \in \F_p} \prod_{j=1}^{n} \left|\mu + (1-\mu) \cos\left(\frac{\pi r a_j}{p}\right) \right|,
\end{align*}
  where the equality holds because the map $\F_p \ni r \mapsto 2r \in \F_p$ is a bijection (as $p$ is odd) and (since $x \mapsto |\cos(\pi x)|$ has period $1$ and it is therefore well defined for $x \in \R/\Z$) because $|\cos(2\pi x/p)| = |\cos(\pi (2x)/p)|$ for every $x \in \F_p$.

At this point, we record the useful inequality 
$$
\left |\mu + (1-\mu)\cos\left(\frac{\pi x}{p}\right)\right| \leq \exp\left(-\frac{1}{2}\left\|\frac{x}{p} \right\|^2\right),
$$
which is valid for every real number $x$ uniformly for all $0\leq \mu \leq 1/2$, where $\|x\| := \|x\|_{\R/\Z}$ denotes the distance to the nearest integer.
Thus, we arrive at
\begin{equation}
\label{eqn:halasz-prelim}
\max_{q\in \F_p}\Pr_{x^{\mu}}\left[\sum_{i=1}^{n} a_i x_i^{\mu} = q\right] \leq \frac{1}{p} \sum_{r\in \F_p} \exp\left(-\frac{1}{2} \sum_{j=1}^{n}\|r a_j/p  \|^2\right).
\end{equation}
Now, for each non-negative real $t$, we define the following `level sets'
$$
T_t := \left\{r \in \F_p : \sum_{j=1}^{n} \|r a_j/p  \|^2 \leq t \right\},
$$
and note that 
\begin{equation}
\label{eqn:halasz-integral}
\sum_{r\in \F_p} \exp\left(-\frac{1}{2} \sum_{j=1}^{n} \|r a_j/p  \|^2\right) = \frac{1}{2}\int_0^{\infty} e^{-t/2}|T_t| dt.
\end{equation}

We will now use a critical estimate due to Hal\'asz. First, note that for any $m\in \N$, the iterated sumset $mT_t$ is contained in $T_{m^2 t}$. Indeed, for $r_1,\dots, r_m \in T_t$, we have from the triangle inequality and the Cauchy-Schwarz inequality that 
\begin{align*}
\sum_{j=1}^{n} \left\| \sum_{i=1}^{m} r_i a_j/ p \right\|^2 \leq \sum_{j=1}^{n} \left(\sum_{i=1}^{m} \left\|r_i a_j/p \right\|\right)^2  \leq \sum_{j=1}^{n} m \sum_{i=1}^{m} \left \|r_i a_j/p\right \|^2  \leq m^2 t.
\end{align*}
Recall that the Cauchy--Davenport theorem states that every pair of nonempty $A, B \subseteq \F_p$ satisfies $|A+B| \ge \min\{ p, |A|+|B|-1\}$. It follows that for every positive integer $m$ and every $t \ge 0$, the iterated sumset $mT_t$ satifies $|mT_t| \ge \min\{ p, m|T_t|-m\}$. Hence, $|T_{m^{2}t}| \geq \min\{ p, m|T_t|-m\}$.

Next, since the map $\F_p \ni r \mapsto ra \in \F_p$ is bijective for every non-zero $a \in \F_p$, we have that
\begin{align*}
\sum_{r \in \F_p} \sum_{j=1}^{n} \|r a_j/p \|^2 &\geq \sum_{j \in \supp(\boldsymbol{a})} \sum_{r \in \F_p} \|r a_j/p\|^2 \\
& = |\supp(\boldsymbol{a})|\sum_{r\in \F_p}\|r/p\|^{2}\\
                             &= \frac{2|\supp(\boldsymbol{a})|}{p^2} \sum_{i=1}^{(p-1)/2} i^2 \\
                             &\geq \frac{|\supp(\boldsymbol{a})|p}{50}.
\end{align*}
On the other hand, it follows from the definition of $T_t$ that for every $t \ge 0$,
\[
  \sum_{r \in \F_p} \sum_{j=1}^n \|r a_j/p \|^2 \le |T_t| \cdot t + \big(p - |T_t|\big) \cdot n.
\]
In particular, we see that $|T_s| < p$ if $s \leq |\supp(\boldsymbol{a})|/100$. Therefore, if $t \leq f(|\supp(\boldsymbol{a})|)$ (as in the statement of the theorem), it follows by setting $m:= \lfloor \sqrt{f(|\supp(\boldsymbol{a})|)/t}\rfloor \geq 1$ that $|T_{m^{2}t}| < p$, and hence, 
\begin{equation}
\label{eqn:halasz-cd}
|T_t|\leq \frac{|T_{m^2t}|}{m} + 1 \leq \frac{2\sqrt{t}|T_{f(|\supp(\boldsymbol{a})|)}|}{\sqrt{f(|\supp(\boldsymbol{a})|)}} + 1.
\end{equation}

We now bound the size of $T_{f(|\supp(\boldsymbol{a})|)}$. Using the elementary inequality $1-100 \|z\|^2 \leq \cos(2 \pi z)
$, which holds for all $z\in \mathbb{R}$, it follows that 
$|T_{f(|\supp(\boldsymbol{a})|)}| \leq |T'|$, where 
$$T' := \left\{r \in \F_p : \sum_{j=1}^n \cos(2 \pi r a_j/p) \geq n - 100 f(|\supp(\boldsymbol{a})|) \right\}.$$
In turn, we will bound the size of $T'$ by computing the moments of the random variable (over the randomness of $r\in \F_p)$ given by $\sum_{j=1}^n \cos\left(\frac{2 \pi r a_j}{p}\right)$. More precisely, by Markov's inequality, we have
for any $\ell \in \N$ that
\begin{equation}
\label{eqn:halasz-moment}
|T'|\leq \frac{1}{\left(n-100f(|\supp(\boldsymbol{a})|)\right)^{2\ell} }\sum_{r\in T'}\left|\sum_{j=1}^{n}\cos\left(\frac{2\pi ra_j}{p}\right)\right|^{2\ell}.
\end{equation}
Moreover, we also have
\begin{align*}
\sum_{r \in T'} \left| \sum_{j=1}^n \cos\left( \frac{2 \pi r a_j}{p}\right)\right|^{2 \ell} &\leq \frac{1}{2^{2 \ell}}\sum_{r\in \F_p} \left| \sum_{j=1}^n (\exp( 2 i \pi r a_j/p) + \exp(- 2 i \pi r a_j/p))\right|^{2\ell} \\
&= \frac{1}{2^{2 \ell}}\sum_{\epsilon_1,\dots, \epsilon_{2\ell}} \sum_{j_1, \dots, j_{2 \ell}} \sum_{r \in \F_p} \exp\left( 2 \pi i r \sum_{i=1}^{2 \ell} \epsilon_i a_{j_i}\right) \\
&= \frac{1}{2^{2\ell}}\sum_{\epsilon_1,\dots,\epsilon_{2\ell}} \sum_{j_1,\dots, j_{2\ell}} p \mathbbm{1}_{\sum_{i=1}^{2 \ell} \epsilon_i a_{j_i} = 0} \\
&\leq \frac{p R_\ell(\boldsymbol{a})}{2^{2 \ell}}.
\end{align*}

Finally, combining this with \cref{eqn:halasz-prelim,eqn:halasz-integral,eqn:halasz-cd,eqn:halasz-moment}, we get for any $0\leq \mu \leq 1/2$ and $k\in \N$ as in the statement of the theorem that
\begin{eqnarray*}
\max_{q\in\F_{p}}\Pr_{x^{\mu}}\left[\sum_{i=1}^{n}a_{i}x_{i}^{\mu}=q\right] & \leq & \frac{1}{2p}\int_{0}^{f(|\supp(\boldsymbol{a})|)}e^{-t/2}|T_{t}|dt+\frac{1}{2}e^{-f(|\supp(\boldsymbol{a})|)/2}\\
 & \leq & \frac{1}{2p}\int_{0}^{f(|\supp(\boldsymbol{a})|)}e^{-t/2}\left(\frac{2\sqrt{t}|T'|}{\sqrt{f(|\supp(\boldsymbol{a})|)}}+1\right)dt+\frac{1}{2}e^{-f(|\supp(\boldsymbol{a})|)/2}\\
 & \leq & \frac{|T'|}{p\sqrt{f(|\supp(\boldsymbol{a})|)}}\int_{0}^{f(|\supp(\boldsymbol{a})|)}e^{-t/2}\sqrt{t}dt+\frac{1}{p}+\frac{1}{2}e^{-f(|\supp(\boldsymbol{a})|)/2}\\
 & \leq & \frac{C_{1}|T'|}{p\sqrt{f(|\supp(\boldsymbol{a})|)}}+\frac{1}{p}+e^{-f(|\supp(\boldsymbol{a})|)/2}\\
 & \leq & \frac{1}{p}+\frac{C_{1}R_{k}(\boldsymbol{a})}{2^{2k}\left(n-100f(|\supp(\boldsymbol{a})|)\right)^{2k}\sqrt{f(|\supp(\boldsymbol{a})|)}}+e^{-f(|\supp(\boldsymbol{a})|)/2}\\
 & \leq & \frac{1}{p}+\frac{CR_{k}(\boldsymbol{a})}{2^{2k}n^{2k}\sqrt{f(|\supp(\boldsymbol{a})|)}}+e^{-f(|\supp(\boldsymbol{a})|)/2},
\end{eqnarray*}
as desired, where the last inequality uses the assumption that $f(|\supp(\boldsymbol{a})|) \leq n/k$.

\end{proof}
\end{document}